\documentclass{amsart}

\usepackage{amsmath,amssymb,latexsym,amsthm,newlfont,enumerate}

\usepackage{graphicx}
\usepackage[all]{xy}

\theoremstyle{plain}

\newtheorem{thm}{Theorem}[section]
\newtheorem{pro}[thm]{Proposition}
\newtheorem{lem}[thm]{Lemma}
\newtheorem{cor}[thm]{Corollary}
\newtheorem{con}[thm]{Conjecture}

\theoremstyle{definition}

\newtheorem{nt}[thm]{Notation}
\newtheorem{rem}[thm]{Remark}
\newtheorem{exa}[thm]{Example}

\newcommand{\Z}{\mathbb{Z}}
\newcommand{\R}{\mathbb{R}}
\newcommand{\Q}{\mathbb{Q}}
\newcommand{\PS}{\mathbb{P}}

\newcommand{\mcal}{\mathcal}
\newcommand{\id}{\mathrm{id}}

\DeclareMathOperator{\inte}{int}
\DeclareMathOperator{\GL}{GL}
\DeclareMathOperator{\Aut}{Aut}
\DeclareMathOperator{\Bir}{Bir}

\DeclareMathOperator{\Pic}{Pic} 
\DeclareMathOperator{\Nef}{Nef}
\DeclareMathOperator{\Amp}{Amp}
\DeclareMathOperator{\Effb}{\overline{\mathrm{Eff}}}
\DeclareMathOperator{\Eff}{\mathrm{Eff}}
\DeclareMathOperator{\Mov}{Mov}
\DeclareMathOperator{\B}{Big}

\begin{document}
\title[On the Cone conjecture]{On the Cone conjecture for Calabi-Yau manifolds with Picard number two}

\author{Vladimir Lazi\'c}
\address{Mathematisches Institut, Universit\"at Bonn, Endenicher Allee 60, 53115 Bonn, Germany}
\email{lazic@math.uni-bonn.de}

\author{Thomas Peternell}
\address{Mathematisches Institut, Universit\"at Bayreuth, 95440 Bayreuth, Germany}
\email{thomas.peternell@uni-bayreuth.de}

\thanks{We thank J.\ Hausen, A.~Prendergast-Smith, D.-Q.~Zhang and the referee for very useful comments and suggestions. We were supported by the DFG-Forschergruppe 790 ``Classification of Algebraic Surfaces and Compact Complex Manifolds".}

\begin{abstract}
Following a recent work of Oguiso, we calculate explicitly the groups of automorphisms and birational automorphisms on a Calabi-Yau manifold with Picard number two. When the group of birational automorphisms is infinite, we prove that the Cone conjecture of Morrison and Kawamata holds.
\end{abstract}

\maketitle
\setcounter{tocdepth}{1}
\tableofcontents

\section{Introduction}

The Cone conjecture of Morrison and Kawamata is concerned with the structure of the nef and the movable cones on a Calabi-Yau manifold in presence of automorphisms or birational automorphisms. To be more precise, consider a Calabi-Yau manifold $X$ with nef cone $\Nef(X)$, the movable cone $\overline{\Mov}(X)$, and effective cone $\Eff(X)$. A Calabi-Yau manifold in our context is a projective manifold $X$ with trivial canonical bundle such that $H^1(X,\mathcal O_X) = 0$. As usual, $\Aut(X)$ respectively $\Bir(X)$ denotes the group of automorphisms respectively birational automorphisms of $X$. Then the Cone conjecture can be stated as follows. 

\begin{con}
Let $X$ be a Calabi-Yau manifold. 
\begin{enumerate}
\item There exists a rational polyhedral cone $\Pi$ which is a fundamental domain for the action of $\Aut(X)$ on $\Nef(X)\cap\Eff(X)$, in the sense that
$$\Nef(X)\cap\Eff(X)=\bigcup_{g\in\Aut(X)}g^*\Pi,$$ 
and $\inte\Pi\cap \inte g^*\Pi=\emptyset$ unless $g^*=\id$.
\item There exists a rational polyhedral cone $\Pi'$ which is a fundamental domain for the action of $\Bir(X)$ on $\overline{\Mov}(X)\cap\Eff(X)$.
\end{enumerate}
\end{con}

There is also the following weaker form.
 
\begin{con}
Let $X$ be a Calabi-Yau manifold. 
\begin{enumerate}
\item There exists a (not necessarily closed) cone $\Pi$ which is a weak fundamental domain for the action of $\Aut(X)$ on $\Nef(X)\cap\Eff(X)$, in the sense that 
$$\Nef(X)\cap\Eff(X)=\bigcup_{g\in\Aut(X)}g^*\Pi,$$
$\inte\Pi\cap \inte g^*\Pi=\emptyset$ unless $g^*=\id$, and for every $g\in\Aut(X)$, the intersection $\Pi\cap g^*\Pi$ is contained in a rational hyperplane.
\item There exists a polyhedral cone $\Pi'$ which is a weak fundamental domain for the action of $\Bir(X)$ on $\overline{\Mov}(X)\cap\Eff(X)$.
\end{enumerate}
\end{con}

For the study of the Cone conjectures, the action 
$$ r\colon \Bir(X) \to \GL(N^1(X)) $$
on the Neron-Severi group $N^1(X)$ is important. We denote by $\mathcal B(X)$ its image, and by $\mathcal A(X)$ the image of the automorphism group. 
\vskip .2cm 
Based on and inspired by recent work of Oguiso \cite{Og12} we prove the following results. 

\begin{thm}
Let $X$ be a Calabi-Yau manifold of Picard number $2$. Then either $|\mcal A(X)|\leq2$, or $\mcal A(X)$ is infinite; and either $|\mcal B(X)|\leq2$, or $\mcal B(X)$ is infinite.
\end{thm}

In fact, we explicitly calculate the groups $\mcal A(X)$ and $\mcal B(X)$, and for more detailed information we refer to Section 3. The consequences for the Cone conjectures can be summarized as follows. 

\begin{thm}
Let $X$ be a Calabi-Yau manifold with Picard number $2$. Then
\begin{enumerate}
\item if the group $\Bir(X)$ is finite, then the weak Cone conjecture holds on $X$;
\item if the group $\Bir(X)$ is infinite, then the Cone conjecture holds on $X$.
\end{enumerate}
\end{thm}

Oguiso in \cite{Og12} showed that there are indeed Calabi-Yau threefolds $X$ with $\rho(X)=2$ and with infinite $\Bir(X)$, as well as hyperk\"ahler $4$-folds $X$ with $\rho(X)=2$ and with infinite $\Aut(X)$. 

\section{Preliminaries}

In this section we give some basic definitions and gather results which we need in the paper. 

\vskip .2cm 
A {\it Calabi-Yau manifold} of dimension $n$ is a projective manifold $X$  with trivial canonical bundle $K_X \simeq \mathcal O_X$
such that $H^1(X,\mathcal O_X) = 0$. In particular, we do not require $X$ to be simply connected.

Let $N^1(X)$ be the Neron-Severi group, generated by the classes of the line bundles on $X$ and let $N^1(X)_\R$ be the corresponding real vector space in
$H^2(X,\mathbb R)$. As usual, ${\Nef}(X) \subseteq N^1(X)_\R$ denotes the closed cone of nef divisors, ${\B}(X)$ stands for the open cone of big divisors, $\overline{\Mov}(X)$ is the closure of the cone generated by mobile divisors (that is, effective divisors whose base locus does not contain divisors), and $\Mov(X)$ is its interior.  Finally, $\Eff(X)$ is the effective cone, and $\Effb(X)$ is the pseudo-effective cone (the closure of the effective cone, or equivalently, the closure of the big cone). 

\vskip .2cm 
On a normal $\mathbb Q$-factorial projective variety $X$ with terminal singularities and nef canonical class, ${\Aut}(X)$ denotes the automorphism group and ${\Bir}(X)$ the group of birational automorphisms. We obtain a natural homomorphism
$$ r\colon \Bir(X) \to {\GL}(N^1(X)) $$
given by $g\mapsto g^*$. 

\begin{nt}\label{notation}
Assume that a Calabi-Yau manifold $X$ has Picard number $\rho(X) = 2$. We let $\ell_1, \ell_2$ be the two boundary rays of ${\Nef}(X)$, and let $m_1,m_2$ be the boundary rays of $\overline{\Mov}(X)$. We fix non-trivial elements $x_i \in \ell_i$ and $y_i \in m_i$. We set 
$$\mathcal A(X) = r\big(\Aut(X)\big)\quad\text{and}\quad \mathcal B(X) = r\big(\Bir(X)\big).$$ 
\end{nt} 

It is well-known, see for instance \cite[Proposition 2.4]{Og12}, that the group $\Bir(X)$ is finite if and only if $\mathcal B(X)$ is, and similarly for $\Aut(X)$ and $\mathcal A(X)$.  

\vskip .2cm 
Recall also the following result \cite[Proposition 3.1]{Og12}.

\begin{pro}\label{proposition:Oguiso}
Let $X$ be a Calabi-Yau manifold  of dimension $n$ such that $\rho(X) = 2$.   
\begin{enumerate} 
\item  If $n$ is odd, or if one of the $\ell_i$ is rational, then every non-trivial element of $\mathcal A(X)$ has order $2$. 
\item If one of the $m_i$ is rational, then every non-trivial element of  $\mathcal B(X) $  has order $2$. 
\end{enumerate} 
\end{pro}

As a consequence, by using Burnside's theorem, Oguiso obtains:

\begin{thm}\label{thm:Oguiso} 
Let $X$ be a Calabi-Yau manifold of dimension $n$ such that $\rho(X) = 2$.   
\begin{enumerate} 
\item If $n$ is odd, then ${\Aut}(X)$ is finite. 
\item If $n$ is even and one of the rays $\ell_i$ is rational, then ${\Aut}(X)$ is finite.
\item If one of the rays $m_i$ is rational, then ${\Bir}(X)$ is finite. 
\end{enumerate}
\end{thm} 

Proposition \ref{automorphism} below makes this result more precise. In contrast to Theorem \ref{thm:Oguiso}, Oguiso constructed an example of Calabi-Yau manifold with $\rho(X) = 2$ such that ${\Bir}(X)$ is infinite. In this example both rays $m_i$ are irrational, and we recall it in Example \ref{exampleOguiso}. 

\vskip .2cm 
If $g$ is any element of $\mcal B(X)$, then $\det g=\pm1$ since $g$ acts on the integral lattice $N^1(X)$. We introduce the notations
$$\mcal A^+(X)=\{g\in\mcal A(X)\mid\det g=1\}$$ 
and 
$$\mcal A^-(X)=\{g\in\mcal A(X)\mid\det g={-}1\};$$ 
and similarly $\mcal B^+(X)$ and $\mcal B^-(X)$.  Note that each $g\in\mcal A(X)$ restricts to an action on the set $\ell_1\cup\ell_2$, and each $g\in\mcal B(X)$ restricts to an action on the set $m_1\cup m_2$. Moreover, since the cone $\Effb(X)$ does not contain lines, this ``restricted" action completely determines $g$. Additionally, each $g \in \mcal A(X)$ is completely determined by $gx_1$ since $\det g=\pm1$. Similarly, each $g\in\mcal B(X)$ is completely determined by $gy_1$.

\vskip .2cm 
We frequently and without explicit mention use the following well-known lemma, see for instance \cite[Lemma 1.5]{Kaw97}.

\begin{lem}
Let $X$ be a Calabi-Yau manifold. Then $g\in\Bir(X)$ is an automorphism if and only if there exists an ample divisor $H$ on $X$ such that $g^*H$ is ample.
\end{lem}

We also use the following result \cite[Theorem 5.7]{Kaw88}, \cite[Corollary 2.7]{Kaw97}, \cite[Theorem 3.8]{KKL12}:

\begin{thm}\label{thm:NefInsideBig}
Let $X$ be a Calabi-Yau manifold. Then the cones $\Nef(X)$ and $\overline{\Mov}(X)$ are locally rational polyhedral in ${\B}(X)$. 
\end{thm}

\section{Calculating $\Aut(X)$ and $\Bir(X)$}
In this section we calculate explicitly the groups $\mcal A(X)$ and $\mcal B(X)$ on a Calabi-Yau manifold with Picard number $2$. We start with some elementary observations. 

\begin{lem}\label{lemma:squaredMinus}
Let $X$ be a Calabi-Yau manifold such that $\rho(X)=2$. If $g\in\mathcal B^-(X)$, then $g^2=\id$. 
\end{lem}
\begin{proof}
By assumption there exist $\alpha>0$ and $\beta>0$ such that $gy_1=\alpha y_2$ and $gy_2=\beta y_1$. But then $g^2y_1=\alpha\beta y_1$ and $g^2y_2=\alpha\beta y_2$, and we have $g^2\in\mcal A^+(X)$. Therefore $\det(g^2)=(\alpha\beta)^2=1$, so $\alpha\beta=1$. Thus, $g^2$ is the identity. 
\end{proof}

\begin{lem}\label{lemma:plusminus}
Let $X$ be a Calabi-Yau manifold such that $\rho(X)=2$. Then $\mcal B^-(X)=\mcal B^+(X)g$ for any $g\in\mcal B^-(X)$. Similarly, $\mcal A^-(X) = \mcal A^+(X)h $ for any $h \in \mcal A^-(X)$. 

In particular, if $\mcal B(X)$ is infinite, so is $\mcal B^+(X)$; and if $\mcal A(X)$ is infinite, so is $\mcal A^+(X)$. 
\end{lem}
\begin{proof}
Let $g,g'\in\mathcal B^-(X)$. Then $g'g=f\in\mcal B^+(X)$, and since $g^2=\id$ by Proposition \ref{proposition:Oguiso}, we have $g'=fg\in\mcal B^+(X)g$. 
The proof in the case of automorphisms is identical.
\end{proof}

\begin{pro}\label{automorphism}
Let $X$ be a Calabi-Yau manifold such that $\rho(X)=2$. If $\mathcal A(X)$ is finite, then $|\mcal A^+(X)|=1$ and $|\mathcal A(X)|\leq2$. If $\mathcal B(X)$ is finite, then $|\mcal B^+(X)|=1$ and $|\mathcal B(X)|\leq2$. 

In particular, if $n$ is odd, or if one of the $\ell_i$ is rational, then $|\mathcal A(X)|\leq2$.
\end{pro}
\begin{proof}
Assume that $\mcal A(X)$ is finite, and fix $g\in\mathcal A(X)$. If $g\in\mcal A^+(X)$, then there exists $\alpha>0$ such that $gx_1=\alpha x_1$. Then $g^m=\id$ for some positive integer $m$, hence $\alpha^m=1$, and therefore $\alpha=1$ and $\mcal A^+(X) =\{\id\}$. Now $|\mcal A(X)|\leq2$ by Lemma \ref{lemma:plusminus}. The proof for $\mcal B(X)$ is the same, and the last claim follows from Theorem \ref{thm:Oguiso}.
\end{proof}

Proposition \ref{automorphism} can also be directly deduced from the following elementary lemma, simplifying calculations in \cite{Og12}.

\begin{lem} \label{easylemma} 
Let $X$ be an $n$-dimensional Calabi-Yau manifold with $\rho(X) = 2$. Assume that $|\mcal A^+(X)| \neq 1 $. Then 
$$ x_1^m \cdot x_2^{n-m}  = 0$$
for all $m$ unless $n = 2m$.

If $n = 2m,$ then $x_1^m \ne 0$ and $x_2^m \ne 0$.
\end{lem} 

\begin{proof} 
Let $f$ be a non-trivial element in $\mathcal A^+$.  Then $fx_1 = \alpha x_1$ and $fx_2 = \alpha^{-1} x_2$ with $\alpha > 0$, $\alpha \ne 1$. 
Then
$$ (fx_1)^m \cdot (fx_2)^{n-m} = \alpha^{2m-n} x_1^m \cdot x_2^{n-m}. $$
On the other hand, 
$$ (fx_1)^m \cdot (fx_2)^{n-m} = x_1^m \cdot x_2^{n-m}, $$
hence $x_1^m \cdot x_2^{n-m} = 0 $ unless $n = 2m$. 

For the second statement, observe that $x_1 + x_2$ is an ample class,
hence 
$$ 0 < (x_1+x_2)^n = {n\choose m} x_1^m \cdot x_2^m, $$ 
and therefore the classes $x_i^m$ are non-zero. 
\end{proof}

\begin{cor}\label{cor:AutInfinite}
Let $X$ be a Calabi-Yau manifold of dimension $n$ such that $\rho(X)=2$. If the group $\Aut(X)$ is infinite, then the following holds.
\begin{enumerate}
\item  $n$ is even and the rays $\ell_i$ are irrational.
\item $\Nef(X)=\Effb(X)$, and $\Nef(X)\cap\Eff(X)=\Amp(X)$. 
\item $c_{n-1}(X)=0$ in $H^{2n-2}(X,\mathbb Q)$.
\end{enumerate}
\end{cor}
\begin{proof}
Claim (1) is Oguiso's Theorem 2.3. 

For the first part of (2), if $\Nef(X)\neq\Effb(X)$, then at least one boundary ray of $\Nef(X)$ is rational by Theorem \ref{thm:NefInsideBig}. This contradicts (1). For the second part of (2), without loss of 
generality it suffices to show that $x_1$ is not effective. Otherwise, we can write $x_1=\sum\delta_j D_j\geq0$ as a sum of at least two prime divisors, since $x_1$ is irrational. But then $\ell_1$ is not an extremal 
ray of the cone $\Nef(X)=\Effb(X)$, a contradiction.

For (3), note that $|\mcal A^+(X)|\geq2$ by Lemma \ref{lemma:plusminus}. Pick a non-trivial element $f\in\mcal A^+(X)$, and let $\alpha\neq1$ be a positive number such that $fx_1=\alpha x_1$. Then 
$$\alpha x_1\cdot c_{n-1}(X)=fx_1\cdot c_{n-1}(X)=x_1\cdot c_{n-1}(X) $$
since the Chern class $c_{n-1}(X)$ is invariant under $f$. Thus $x_1\cdot c_{n-1}(X)=0$; similarly we get $x_2\cdot c_{n-1}(X)=0$. Therefore $c_{n-1}(X)=0$ as $\{x_1,x_2\}$ is a basis of $N^1_\R(X)$.
\end{proof}

\begin{rem} 
(1) The same arguments as in Corollary \ref{cor:AutInfinite} yield 
$$ c_{i_1}(X) \cdot \ldots \cdot c_{i_r}(X) = 0$$
if $i_1 + \ldots + i_r = n-1$. \\
(2) We do not know of any example of a simply connected Calabi-Yau manifold $X$ in the strong sense (i.e.\ such that $H^q(X,\mathcal O_X) = 0$  for $1 \leq q \leq n-1$) of even dimension $n$ such that $c_{n-1}(X) = 0$. 
One might wonder whether any simply connected irreducible projective manifold $X$ of dimension $n$ with $\omega_X\simeq \mathcal O_X$ and $c_{n-1}(X) = 0$ is a hyperk\"ahler manifold.
\end{rem} 
 
In some further cases, the even dimensional case can be treated: 

\begin{thm} Let $X$ be a Calabi-Yau manifold of even dimension  $n$. If $\rho(X) = 2$ and if $c_2(X) $ can be represented by a positive closed
$(2,2)$-form, then ${\Aut}(X)$ is  finite. 
\end{thm} 

\begin{proof} 
Arguing by contradiction, we suppose that there is an automorphism $f\in\mcal A^+(X)$ of infinite order, cf.\ Lemma \ref{lemma:plusminus}. Write $n = 2m$. Then $x_1^m \ne 0$ and $x_2^m  \ne 0$ by Lemma \ref{easylemma}. 

Suppose that $m$ is even, and write $m = 2k$. Then
$$ x_1^{2k}\cdot c_2(X)^k  > 0 $$
by our positivity assumption on $c_2(X)$. On the other hand, 
$$ x_1^{2k} \cdot c_2(X)^k = (fx_1)^{2k} \cdot c_2(X)^k = \alpha^{2k} x_1^{2k} \cdot c_2(X)^k $$
since $c_2(X)$ is invariant under $f$.  Since $\alpha \neq 1$, this is a contradiction. 

If $m $ is odd, we write $n = 4s+ 2$ and argue with $x_1^{2s} \cdot c_2(X)^{s+1}$. 
\end{proof} 

Notice that for every projective manifold $X$ of dimension $n$ with nef canonical bundle, the second Chern class $c_2(X)$ has the following positivity property (Miyaoka \cite{Miy87}):
$$ c_2(X) \cdot H_1 \ldots \cdot H_{n-2} \geq 0 $$
for all ample line bundles $H_j$.
\vskip .2cm   
Concerning bounds for $\mathcal B(X)$, we have:

\begin{pro}\label{proposition:birgroupBoundary}
Let $X$ be a Calabi-Yau manifold such that $\rho(X)=2$. Assume that $\Nef(X)\nsubseteq{\Mov}(X)$. Then $\mathcal A^+(X)=\mathcal B^+(X)$. In particular, if the dimension of $X$ is odd, then $|\mathcal B(X)|\leq2$.
\end{pro}
\begin{proof} 
The condition $\Nef(X) \nsubseteq \Mov(X)$ implies that one of the rays $\ell_i$ is an extremal ray of $\overline{\Mov}(X)$. Hence, without loss of generality, we may assume that $m_1=\ell_1$. Let $g$ be a non-trivial element of $\mathcal B^+(X)$. Then $g\ell_1=gm_1=m_1$, and $m_1$ is an extremal ray of the cone 
$$\R_+m_1+\R_+g\ell_2=\R_+g\ell_1+\R_+g\ell_2=g\Nef(X).$$
This implies that $g\Nef(X)$ intersects the interior of $\Nef(X)$, and hence $g\in\mcal A(X)$ by Lemma 2.4. This proves the first claim.

The second claim then follows from Proposition \ref{automorphism}.
\end{proof}

\begin{thm}\label{automorphismDet1}
Let $X$ be a Calabi-Yau manifold such that $\rho(X)=2$. Then either $|\mcal A^+(X)|=1$ or $\mcal A^+(X)\simeq\Z$; and either $|\mcal B^+(X)|=1$ or $\mcal B^+(X)\simeq\Z$.
\end{thm}
\begin{proof}
Assume that $|\mcal A^+(X)|\geq2$. For every $g\in\mcal A^+(X)$, let $\alpha_g$ be the positive number such that $gy_1=\alpha_g y_1$, and set 
$$\mcal S=\{\alpha_g\mid g\in\mcal A^+(X)\}.$$
Note that $\mcal S$ is a multiplicative subgroup of $\R^*$ and that the map 
$$ \mcal A^+(X) \to \mcal S, \quad g \mapsto \alpha_g$$
is an isomorphism of groups. We need to show that $\mcal S$ is an infinite cyclic group. 

We first show that $\mcal S$ is, as a set, bounded away from $1$. Otherwise, we can pick a sequence $(g_i)$ in $\mcal A^+(X)$ such that $\alpha_{g_i}$ converges to $1$. Fix two integral linearly independent classes $h_1$ and $h_2$ in $N^1(X)_\R$. Then $g_ih_1$ converge to $h_1$ and $g_ih_2$ converge to $h_2$. Since $g_ih_1$ and $g_ih_2$ are also integral classes and $N^1(X)$ is a lattice in $N^1(X)_\R$, this implies that $g_ih_1=h_1$ and $g_ih_2=h_2$ for $i\gg0$, and hence $g_i=\id$ for $i\gg0$. 

Hence, the set $\mcal S'=\{\ln\alpha\mid\alpha\in\mcal S\}$ is an additive subgroup of $\R$ which is discrete as a set. Then it is a standard fact that $\mcal S'$, and hence $\mcal S$, is isomorphic to $\Z$, cf.\ \cite[21.1]{Fo81}.

The proof for the birational automorphism group is the same.
\end{proof}

\section{Structures of $\Nef(X)$ and $\overline{\Mov}(X)$}

\begin{pro}\label{proposition:weakKMconjecture}
Let $X$ be a Calabi-Yau manifold such that $\rho(X)=2$. If $\mcal A(X)$ is finite, then the weak Cone conjecture holds for $\Nef(X)$. If $\mcal B(X)$ is finite, then the weak Cone conjecture holds for $\overline{\Mov}(X)$.
\end{pro}
\begin{proof}
We only prove the statement about the nef cone, since the other statement is analogous. By Proposition \ref{automorphism}, we have $\vert \mcal A(X) \vert \leq  2$, hence we may assume that $|\mathcal A(X)|=2$. Fix an integral class $x\in\Nef(X)$, let $g\in\mathcal A^-(X)$, and consider the class $y=x+gx\in\Nef(X)$. Then $y$ is fixed under the action of $\mathcal A(X)$. Since $g$ acts on $N^1(X)$, both $gx$ and $y$ must be integral. It is then obvious that $\Pi=\ell_1+\R_+y$ is a fundamental domain for the action of $\mathcal A(X)$ on $\Nef(X)$.
\end{proof}

\begin{rem}
If $X$ is a Calabi-Yau manifold of odd dimension such that $\rho(X)=2$ and $\Nef(X)\nsubseteq{\Mov}(X)$, then the weak Cone conjecture holds for $\overline{\Mov}(X)$. The proof is analogous to that of Proposition \ref{proposition:weakKMconjecture}, using Proposition \ref{proposition:birgroupBoundary}.
\end{rem}

\begin{pro}\label{proposition:NefInMov}
Let $X$ be a Calabi-Yau manifold such that $\rho(X)=2$. Assume that $\Nef(X)\subseteq{\Mov}(X)$. Then the Cone conjecture holds for $\Nef(X)$.
\end{pro}
\begin{proof}
By assumption, we have $\Nef(X)\subseteq\B(X)$, and hence, the nef cone is rational polyhedral by Theorem \ref{thm:NefInsideBig}. Then argue as in the proof of Proposition \ref{proposition:weakKMconjecture}.
\end{proof}

\begin{lem}\label{lem:BirInfiniteMov}
Let $X$ be a Calabi-Yau manifold with $\rho(X)=2$. Assume that $\Bir(X)$ is infinite. Then $\overline{\Mov}(X)\cap\Eff(X)=\Mov(X)$.
\end{lem}
\begin{proof}
The rays of $\overline{\Mov}(X)$ are irrational by Proposition \ref{proposition:Oguiso}, and therefore $\overline{\Mov}(X)=\Effb(X)$ by Theorem \ref{thm:NefInsideBig}. 
We cannot have $y_1\in\Eff(X)$: otherwise, we can write $y_1=\sum\delta_i D_i\geq0$ as a sum of at least two different prime divisors, since $m_1$ is irrational. But then $m_1$ is not an extremal ray of the cone $\overline{\Mov}(X)=\Effb(X)$, a contradiction. This concludes the proof.
\end{proof}

\begin{thm}\label{thm:BirInfiniteKM}
Let $X$ be a Calabi-Yau manifold with $\rho(X)=2$. If the group $\Bir(X)$ is infinite, then the Cone conjecture holds on $X$.
\end{thm}
\begin{proof} 
(i) First we show that the Cone conjecture holds for $\Nef(X)$ in case $\Aut(X)$ is infinite. 

Note that $\Nef(X)=\Effb(X)$ and $\Nef(X)\cap\Eff(X)=\Amp(X)$ by Corollary \ref{cor:AutInfinite}(2), and in particular we have $\mcal A(X)=\mcal B(X)$. By Lemma \ref{lemma:plusminus} and Theorem \ref{automorphismDet1}, we know that $\mcal A(X)=\mcal A^+(X)\cup\mcal A^-(X)$, where $\mcal A^+(X)\simeq\Z$ and $\mcal A^-(X)=\mcal A^+(X)g$ for any $g\in\mcal A^-(X)$. 
 
Assume first that $\mcal A(X)=\mcal A^+(X)\simeq\Z$. Let $h$ be a generator of $\mcal A(X)$, let $x$ be any point in $\Amp(X)$, and denote  
$$\Pi=\R_+x+\R_+hx.$$ 
It is then straightforward to check that $\Pi$ is a fundamental domain for the action of $\mcal A(X)$ on $\Amp(X)$. Indeed, it is clear that the cones $h^k\Pi$ have disjoint interiors, and to see that they cover $\Amp(X)$, it suffices to notice that the rays $\R_+h^kx$ converge to $\ell_1$, respectively $\ell_2$, when $k\to\pm\infty$. 

Now assume that $\mcal A^-(X)\neq\emptyset$. Let $f$ be a generator of $\mcal A^+(X)$, let $\tau$ be an element of $\mcal A^-(X)$, and let $x$ be an integral class in $\Amp(X)$. Set 
$$z_1=x+\tau x\quad\text{and}\quad z_2=z_1+fz_1,$$
and note that $z_1$ and $z_2$ are integral classes since $\tau$ and $f$ act on $N^1(X)$. Denote $\theta=f\tau\in\mcal A^-(X)$. Then $\tau^2=\theta^2=\id$ by Lemma \ref{lemma:squaredMinus}, and hence
$$\theta\tau=(f\tau)\tau=f\qquad\text{and}\qquad\theta f=\theta(\theta\tau)=\tau.$$
This implies
\begin{equation}\label{eq:rays}
\tau z_1=z_1,\qquad \theta z_1=fz_1,\qquad \theta z_2=z_2. 
\end{equation} 

Now, let 
$$\Pi=\R_+z_1+\R_+z_2.$$  
Then $\Pi$ is a rational polyhedral cone, and we claim that $\Pi$ is a fundamental domain for the action of $\mcal A(X)$ on $\Amp(X)$. 
\begin{figure}[htb]
\begin{center}
\includegraphics[width=0.41\textwidth]{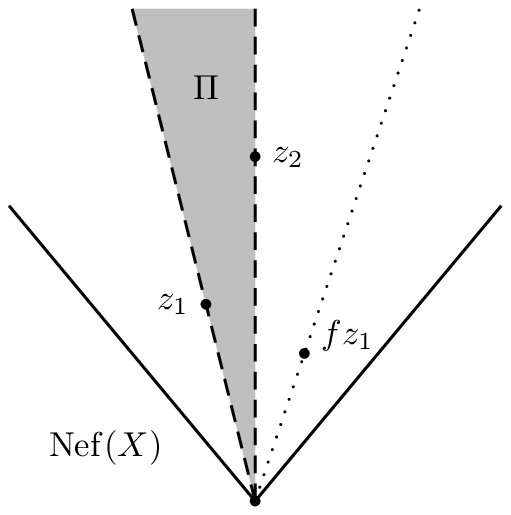}
\end{center}
\end{figure}

First, by \eqref{eq:rays} we have 
$$\theta\Pi=\R_+\theta z_1+\R_+\theta z_2=\R_+fz_1+\R_+z_2,$$ 
and thus
$$\Pi\cup\theta\Pi=\R_+z_1+\R_+fz_1.$$
This implies 
$$\bigcup_{k\in\Z}f^k(\Pi\cup\theta\Pi)=\Amp(X)$$ 
as in the first part of the proof, and therefore,  
$$\bigcup_{g\in\mcal A(X)}g\Pi=\Amp(X).$$ 

Second, assume that there exists $\lambda\in\mcal A(X)$ such that $\inte\Pi\cap\inte\lambda\Pi\neq\emptyset$. Then, possibly after replacing $\lambda$ by $\lambda^{-1}$, this implies that $\lambda z_1\subseteq\inte\Pi$ or $\lambda z_2\subseteq\inte\Pi$. If $\lambda z_1\subseteq\inte\Pi$, then by Lemma \ref{lemma:plusminus} there exists $k\in\Z$ such that $\lambda=f^k\tau$, hence $\lambda z_1=f^k z_1\in\inte\Pi$ by \eqref{eq:rays}, which is clearly impossible. Similarly, if $\lambda z_2\subseteq\inte\Pi$, again by Lemma \ref{lemma:plusminus} there exists $\ell\in\Z$ such that $\lambda=f^\ell\theta$, hence $\lambda z_2=f^\ell z_2\in\inte\Pi$ by \eqref{eq:rays}, a contradiction. This finishes the proof of (i).
\vskip .2cm 

(ii) Next we show that the Cone conjecture holds for $\Nef(X)$ if $\Aut(X)$ is finite but $\Bir(X)$ is infinite. Here $\Nef(X)\subseteq{\Mov}(X)$ by Lemma \ref{lemma:plusminus} and Proposition \ref{proposition:birgroupBoundary}. 
Then the Cone conjecture for $\Nef(X)$ holds by Proposition \ref{proposition:NefInMov}. 
\vskip .2cm 

(iii) Finally,
note that $\overline{\Mov}(X)\cap\Eff(X)=\Mov(X)$ by Lemma \ref{lem:BirInfiniteMov}, hence the proof of the Cone conjecture for $\Mov(X)$ is the same as that of (i) by a simple adaption. 
\end{proof}

\begin{exa}\label{exampleOguiso}
We recall \cite[Proposition 6.1]{Og12}. Oguiso constructs a Calabi-Yau $3$-fold $X$ with Picard number $2$, obtained as the intersection of general hypersurfaces in $\PS^3\times\PS^3$ of bi-degrees $(1,1)$, $(1,1)$, and $(2,2)$, which has the following properties: $x_1$ and $x_2$ are rational, $y_1=(3+2\sqrt2)x_2-x_1$, $y_2=(3+2\sqrt2)x_1-x_2$, there are two birational involutions $\tau_1$ and $\tau_2$ such that $\tau_1\tau_2$ is of infinite order, and the group $\Bir(X)$ is generated by $\Aut(X)$ and by $\tau_1$ and $\tau_2$.
\end{exa}

We now show that Example \ref{exampleOguiso} is a typical example of a Calabi-Yau manifold with Picard number $2$ and with infinite group of birational automorphisms.

\begin{thm} \label{unique}
Let $X$ be a Calabi-Yau manifold of dimension $n$ and with $\rho(X)=2$. Assume that $\Bir(X)$ is infinite.
\begin{enumerate}
\item Let $f$ be a generator of $\mcal B^+(X)$, and let $\alpha>0$ be the real number such that $fy_1=\alpha y_1$. Then $[\Q(\alpha):\Q]=2$.
\item Let $\{v,w\}$ be any integral basis of $N^1(X)_\R$. Then $m_1=\R_+(av+bw)$ and $m_2=\R_+(cv+dw)$, where $a,b,c,d\in\Q(\alpha)$.
\item There exist a birational automorphism $\tau$ (possibly the identity) such that $\tau^2\in\Aut(X)$, and a birational automorphism of infinite order $\sigma$ such that the group $\Bir(X)$ is generated by $\Aut(X)$ and by $\tau$ and $\sigma$.
\end{enumerate}
\end{thm}
\begin{proof}
By rescaling $y_1$ and $y_2$, we can assume that
$$h=y_1+y_2$$ 
is a primitive integral class in $N^1(X)_\R$. Denote
$$h'=fh=\alpha y_1+\frac1\alpha y_2\quad\text{and}\quad h''=f^2h=\alpha^2 y_1+\frac1{\alpha^2} y_2;$$
these are again primitive integral classes since $\mcal B(X)$ preserves $N^1(X)$. Then an easy calculation shows that
$$h+ h''=\frac{\alpha^2+1}{\alpha}h',$$
and hence the number $\frac{\alpha^2+1}{\alpha}=\alpha+\frac1\alpha$ is an integer. Since 
$$y_1=\frac1{\alpha^2-1}(\alpha h'-h),$$
and $y_1$ is not rational by Theorem \ref{thm:Oguiso}, the number $\alpha$ cannot be rational, and (1) follows.
\vskip .2cm
For (2) fix an integral basis $\{v,w\}$ of $N^1(X)_\R$, and write 
$$y_1=av+bw\quad\text{and}\quad y_2=cv+dw.$$
Then 
$$h=(a+c)v + (b+d)w\quad\text{and}\quad h'=(\alpha a+c/\alpha)v+(\alpha b+d/\alpha)w.$$
Write $p=a+c$ and $q=\alpha a+c/\alpha$, and note that $p,q\in\Z$. Then an easy calculation shows that $a,c\in\Q(\alpha)$, and similarly for $b$ and $d$.\vskip .2cm 
Finally, for (3), note that by Theorem \ref{automorphismDet1} and Lemma \ref{lemma:plusminus}, we have $\mcal B(X)=\mcal B^+(X)\cup\mcal B^-(X)$, where $\mcal B^+(X)$ is infinite 
cyclic with generator $\sigma'$, and $\mcal B^-(X)=\mcal B^+(X)\tau'$ for any $\tau'\in\mcal B^-(X)$. Pick $\tau,\sigma\in\Bir(X)$ such that 
$$r(\tau)=\tau'\quad\text{and}\quad r(\sigma)=\sigma',$$
see Notation \ref{notation}. Since $r(\tau^2)=\tau'^2=\id$ by Lemma \ref{lemma:squaredMinus}, it follows that $\tau^2$ is an isomorphism by \cite[Proposition 2.4]{Og12}. Now take an element $\theta$ is any element of $\Bir(X)$, 
then there exist integers $k$ and $\ell$ such that $r(\theta)=\sigma'^k\tau'^\ell=r(\sigma^k\tau^\ell)$, and we conclude again by \cite[Proposition 2.4]{Og12}.
\end{proof}

\begin{rem}
We are indebted to the referee for pointing out the following example, which provides a variety satisfying the assumptions of Theorem \ref{unique} in any dimension $n\geq 3$.

Let $X$ be the complete intersection
$$H_1 \cap H_2 \cap\dots\cap H_{n-1}\cap Q \subseteq \mathbb P^n \times \mathbb P^n,$$
where $n\geq 3$, where $H_i$ are general hypersurfaces of bidegree $(1, 1)$, and where $Q$ is a general hypersurface of bidegree $(2, 2)$. Then $X$ is a simply connected Calabi-Yau $n$-fold with Picard number two. More precisely, $\Pic (X) = \Z L_1 \oplus \Z L_2$, where $L_1$ and $L_2$ are pullbacks of the hyperplane classes of factors $\mathbb P^n$. Consider the two birational involutions $\iota_1, \iota_2$ induced by the two natural projections of $X$ to $\mathbb P^n$. Then $\iota_1\iota_2$ is a birational automorphism of $X$ of infinite order. The last statement can be checked by computing $(\iota_1\iota_2)^*L_i$ as in \cite[Proposition 6.1]{Og12}. 
\end{rem}

\begin{rem}
One can obtain a similar description of the cone $\Nef(X)$ when the automorphism group of $X$ is infinite.

Basically there are two types of simply connected irreducible Calabi-Yau manifolds: those which do not carry any holomorphic forms of intermediate degree -- these manifolds are often simply called Calabi-Yau manifolds -- and hyperk\"ahler manifolds carrying a non-degenerate holomorphic $2$-form. While in the hyperk\"ahler case the nef cone can be irrational by \cite[Proposition 1.3]{Og12}, it is believed that the nef cone 
of a ``strict'' Calabi-Yau manifold with, say, $\rho(X) = 2$, must be rational. The evidence is provided by the fact that in odd dimensions $\Aut(X)$ is finite, and then the Cone conjecture would imply
the rationality. In even dimensions we saw that an infinite automorphism group on a strict Calabi-Yau manifold with Picard number two is possible only in very special circumstances. 
\end{rem}

\bibliographystyle{amsalpha}

\bibliography{biblio}
\end{document}